\documentclass[leqno,11pt]{amsart}

\usepackage{amssymb}
\usepackage{amsmath}
\usepackage{enumerate}
\usepackage{color}
\usepackage[pdftex]{graphicx}

\def\s{\mathbb{S}}

\def\R{\mathbb{R}}
\def\r{\mathbb{R}}

\def\z{\mathbb{Z}}

\def\C{\mathbb{C}}
\def\c{\mathbb{C}}

\newcommand{\im}{\mathop{\rm Im }\nolimits}

\newtheorem{remark}{Remark}[section]
\newtheorem{theorem}{Theorem}[section]
\newtheorem{proposition}{Proposition}[section]
\newtheorem{corollary}{Corollary}[section]

\numberwithin{equation}{section}

\begin{document}

\title[A new construction of Lagrangians in terms of planar curves]{A new construction of Lagrangians \\ in the complex Euclidean plane \\ in terms of planar curves}

\author{Ildefonso Castro}
\address{Departamento de Matem\'{a}ticas \\
Universidad de Ja\'{e}n \\
23071 Ja\'{e}n, SPAIN} \email{icastro@ujaen.es}

\author{Ana M.~Lerma}
\address{Departamento de Matem\'{a}ticas \\
Universidad de Ja\'{e}n \\
23071 Ja\'{e}n, SPAIN} \email{alerma@ujaen.es}

\thanks{Research partially supported by a MEC-Feder grant
MTM2011-22547 and a Junta Andalucia Grant P09-FQM 5088}



\date{}

\begin{abstract}
We introduce a new method to construct a large family of Lagrangian surfaces in complex Euclidean plane $\C^2$ by means
of two planar curves making use of their usual product as complex functions and integrating the Hermitian product of
their position and tangent vectors.

Among this family, we characterize minimal, constant mean curvature, Hamiltonian stationary, solitons for mean
curvature flow and Willmore surfaces in terms of simple properties of the curvatures of the generating curves. As an
application, we provide explicitly conformal parametrizations of known and new examples of these classes of Lagrangians
in $\C^2$.
\end{abstract}

\maketitle


\section{Introduction}

An isometric immersion $\phi:M^n\rightarrow \widetilde{M}^n$ of an $n$-dimensional Riemannian manifold $M^n$ into an
$n$-dimensional Kaehler manifold $\widetilde{M}^n$ is said to be Lagrangian if the complex structure $J$ of
$\widetilde{M}^n$ interchanges each tangent space of $M^n$ with its corresponding normal space. Lagrangian submanifolds
appear naturally in several contexts of Mathematical Physics. For example, special Lagrangian submanifolds of the
complex Euclidean space $\C^n$ (or of a Calabi-Yau manifold) have been studied widely and in \cite{SYZ} it was proposed
an explanation of mirror symmetry in terms of the moduli spaces of special Lagrangian submanifolds. These submanifolds
are volume minimizing and, in particular, they are minimal submanifolds. In the two-dimensional case, special
Lagrangian surfaces of $\C^2$ are exactly complex surfaces with respect to another complex structure on $\R^4\equiv
\C^2$.

The simplest examples of Lagrangian surfaces in $\C^2$ are given by the product of two planar curves
$\alpha=\alpha(t)$, $t\in I_1\subseteq\R$, and $\omega=\omega(s)$, $s\in I_2\subseteq\R$:
\begin{align}\label{times-curves}
(t,s)\stackrel{\phi}{\longmapsto}\left(\alpha(t),\omega(s)\right).
\end{align}

Another fruitful method of construction of Lagrangian surfaces in $\C^2$ is obtained when one takes the particular
version for the two-dimensional case of Proposition 3 in \cite{RU} (see also \cite{CM} and \cite{AC}), involving a
planar curve $\alpha=\alpha(t)$, $t\in I_1\subseteq\R$, and a Legendre curve $(\gamma_1,\gamma_2)=\gamma=\gamma(s)$,
$s\in I_2\subseteq\R$, in the 3-sphere $\mathbb{S}^3\subset\C^2$
\begin{align}\label{cdot-curves}
(t,s)\stackrel{\phi}{\longmapsto}\alpha(t)\cdot\gamma(s)=\alpha(t)\left(\gamma_1(s),\gamma_2(s)\right).
\end{align}

In \cite{CCh}, it was presented a different method to construct a large family of Lagrangian surfaces in $\C^2$ using a Legendre
curve $(\alpha_1,\alpha_2)=\alpha(t)$, $t\in I_1\subseteq\R$, in the anti De Sitter 3-space $\mathbb{H}^3_1\subset\C^2$ and a
Legendre curve $(\gamma_1,\gamma_2)=\gamma(s)$, $s\in I_2\subseteq\R$, in $\mathbb{S}^3\subset\C^2$:
\begin{align}\label{circ-curves}
(t,s)\stackrel{\phi}{\longmapsto}\left(\alpha_1(t)\gamma_1(s),\alpha_2(t)\gamma_2(s)\right).
\end{align}

We observe that in the constructions \eqref{times-curves}, \eqref{cdot-curves} and \eqref{circ-curves} the components of the position
vector $\phi=(\phi_1,\phi_2)$ of the immersions are given by the product of two complex functions:
\begin{align}
\phi_1(t,s)=\begin{cases}\alpha(t)\\\alpha(t)\gamma_1(s)\\\alpha_1(t)\gamma_1(s)\end{cases},\qquad\phi_2(t,s)
\begin{cases}\omega(s)\\\alpha(t)\gamma_2(s)\\\alpha_2(t)\gamma_2(s)\end{cases}.
\end{align}

From an algebraic point of view, we propose now to consider one of the components as a product of two complex functions
and the other as an addition of another two complex functions. So, we can consider the following type of possible
Lagrangian immersions:
\begin{align}\label{ansatz}
\phi(t,s)=\left(f(t)+g(s),\alpha(t)\omega(s)\right),
\end{align}
where $\alpha=\alpha(t)$, $t\in I_1\subseteq\R$ and $\omega=\omega(s)$, $s\in I_2\subseteq\R$ are planar curves. If we
impose that $\phi$ gives an orthonormal parametrization of a Lagrangian immersion, we have that
$\left(\phi_t,\phi_s\right)=0$, where $(\cdot,\cdot)$ denotes the usual bilinear Hermitian product of $\C^2$. Since
$\phi_t=(f'(t),\alpha'(t)\omega(s))$ and $\phi_s=(\dot g(s),\alpha(t)\dot\omega(s))$ where $'$  (resp.\ $\dot\ $) means
derivate respect to $t$ (resp.\ to $s$), we get
\begin{align}\label{condition*}
f'(t)\overline{\dot g(s)}+\alpha'(t)\overline{\alpha(t)}\omega(s)\overline{\dot\omega(s)}=0.
\end{align}
So, essentially we can take
\begin{align}
f(t)=-\int\alpha'(t)\overline{\alpha(t)}dt,\qquad g(s)=\int\dot\omega(s)\overline{\omega(s)}ds.
\end{align}
Putting this in \eqref{ansatz} we can check that
\begin{align}\label{alfaomega}
\phi(t,s)=\left(\int\dot\omega(s)\overline{\omega(s)}ds-\int\alpha'(t)\overline{\alpha(t)}dt,\alpha(t)\omega(s)\right)
\end{align}
is a Lagrangian immersion constructed from two planar curves (see Theorem \ref{thm_immersion}), well defined up to a
translation in $\C^2$.

An interesting problem in this setting is to find nontrivial examples of Lagrangian surfaces with some given geometric
properties. In this paper we pay our attention to an extrinsic point of view and focus on several classical equations
involving the mean curvature vector and natural associated variational problems. In this way, we determine in our
construction of Lagrangians not only those which are minimal, have parallel mean curvature vector or constant mean
curvature, but also those ones that are Hamiltonian stationary, solitons of mean curvature flow or Willmore.

When we involve lines and circles in (\ref{alfaomega}) we get the most regular surfaces: special Lagrangians (Corollary
\ref{cor_SpecialLagrangians}) and Hamiltonian stationary Lagrangians (Corollary \ref{cor_HSL}). In this setting, we
provide explicit conformal parametrizations of some known examples in terms of elementary functions and obtain some new
examples of interesting Hamiltonian stationary Lagrangians. With some more sophisticated curves, we obtain a very large
family of new Lagrangians with constant mean curvature vector (Corollary \ref{cor_Hconstant}), which includes a
(branched) Lagrangian torus. Our construction (\ref{alfaomega}) is actually inspired in the Lagrangian translating
solitons obtained in \cite{CL12} associated to certain special solutions of the curve shortening flow that we recover
in Corollary \ref{cor_translating}. We also provide Willmore Lagrangians when we consider free elastic curves
(Corollary \ref{cor_willmore}). Finally, we illustrate in section 3.8 that we can also arrive at Lagrangian tori
starting from certain closed curves.

The key point of the proof of all our results is the simple expression (\ref{MeanCurvatureVector}) for the mean
curvature vector of the Lagrangian immersion in terms of the curvature functions of the generatrix curves.
\vspace{0.5cm}


\section{The construction}
In the complex Euclidean plane $\c^2$ we consider the bilinear
Hermitian product defined by
\[
(z,w)=z_1\bar{w}_1+z_2\bar{w}_2, \quad  z,w\in\c^2.
\]
Then $\langle\, \, , \, \rangle = {\rm Re} (\,\, , \,)$ is the
Euclidean metric on $\c^2$ and $\omega = -{\rm Im} (\,,)$ is the
Kaehler two-form given by $\omega (\,\cdot\, ,\,\cdot\,)=\langle
J\cdot,\cdot\rangle$, where $J$ is the complex structure on
$\c^2$.

Let $\phi:M \rightarrow \c^2$ be an isometric immersion of a
surface $M$ into $\c^2$. $\phi $ is said to be Lagrangian if
$\phi^* \omega = 0$.  Then we have $\phi^* T\c^2 =\phi_* TM \oplus
J \phi_* T M$, where $TM$ is the tangent bundle of $M$. The second
fundamental form $\sigma $ of $\phi $ is given by $\sigma
(v,w)=JA_{Jv}w$, where $A$ is the shape operator, and so the
trilinear form $$C(\cdot,\cdot,\cdot)=\langle \sigma(\cdot,\cdot),
J \cdot \rangle $$ is fully symmetric.

Suppose $M$ is orientable and let $\omega_M$ be the area form of
$M$. If $\Omega = dz_1 \wedge dz_2$ is the closed complex-valued
2-form of $\c^2$, then $\phi^* \Omega = e^{i\beta} \omega_M$,
where $\beta:M\rightarrow \r /2\pi \z$ is called the {\em
Lagrangian angle} map of $\phi$ (see \cite{HL}).  In general,
$\beta $ is a multivalued function; nevertheless $d\beta $ is a
well defined closed 1-form on $M$ and its cohomology class is
called the Maslov class.

It is remarkable that $\beta$ satisfies (see for example
\cite{SW})
\begin{equation}\label{beta}
H=J\nabla\beta,
\end{equation}
where $H$ is the mean curvature vector of $\phi$, defined by
$H={\rm trace} \, \sigma$.

In this section, we describe a new method to contruct Lagrangian
surfaces in complex Euclidean plane $\c^2$ with nice geometric
properties, in the sense of they are similar to those of a product
of curves.
\begin{theorem}\label{thm_immersion}
Let $\alpha=\alpha(t)\subset\c\setminus\{0\}$, $t\in I_1$, and
$\omega=\omega(s)\subset\c\setminus\{0\}$, $s\in I_2$, be regular
planar curves, where $I_1$ and $I_2$ are intervals of $\r$. For
any $t_0\in I_1$ and $s_0\in I_2$, let define
\begin{equation*}
\Phi=\alpha \ast \omega :I_1\times
I_2\subset\r^2\rightarrow\c^2=\c \times \c
\end{equation*}
\begin{equation*}
\Phi(t,s)=\left(\int_{s_0}^s\dot\omega(y)\overline{\omega(y)}dy-\int_{t_0}^t\alpha'(x)\overline{\alpha(x)}dx
\,  ,\, \alpha(t)\omega(s)\right).
\end{equation*}

Then $\Phi$ is a Lagrangian immersion whose induced metric is
\begin{align}\label{metricPhi}
\Phi^*\langle\ ,\
\rangle=\left(|\alpha|^2+|\omega|^2\right)\left(|\alpha'|^2dt^2+|\dot\omega|^2ds^2\right),
\end{align}
where $'$ and $^\cdot$ denote the derivatives respect to $t$ and
$s$ respectively.

The intrinsic tensor $ C(\cdot,\cdot,\cdot)=\langle
\sigma(\cdot,\cdot),J \cdot \rangle $ of $\Phi=\alpha \ast \omega$
is given by
\begin{align}\label{TensorC}
C(\partial_t,\partial_t,\partial_t)&=|\alpha '|^2 \left(
(|\alpha|^2+|\omega|^2)|\alpha '|
\kappa_\alpha - \langle  \alpha ' , J\alpha \rangle \right) \\
C(\partial_t,\partial_t,\partial_s)&= |\alpha '|^2 \langle \dot \omega , J \omega \rangle \nonumber \\
C(\partial_t,\partial_s,\partial_s)&= |\dot \omega|^2 \langle \alpha ' , J \alpha \rangle \nonumber \\
C(\partial_s,\partial_s,\partial_s)&= |\dot \omega |^2 \left(
(|\alpha|^2+|\omega|^2)|\dot \omega| \kappa_\omega - \langle  \dot
\omega , J\omega  \rangle \right) \nonumber
\end{align}
where $\kappa_\alpha $ and $\kappa_\omega$ are the curvature
functions of $\alpha $ and $\omega$, and $J$ also denotes the
$+\pi/2$-rotation acting on $\c\equiv \r^2$.

The Lagrangian angle map of $\Phi=\alpha \ast \omega$ is given by
\begin{align}\label{LagrangianAngle}
\beta =\arg\left(\alpha'\right)+\arg\left(\dot\omega\right)+\pi
\end{align}
and the mean curvature vector $H$ of $\Phi=\alpha \ast \omega$ by
\begin{align}\label{MeanCurvatureVector}
H
=\frac{1}{|\alpha|^2+|\omega|^2}\left(\frac{\kappa_\alpha}{|\alpha
' |} J\Phi_t+\frac{\kappa_\omega}{|\dot \omega|} J\Phi_s\right).
\end{align}
\end{theorem}

\begin{proof}
We first compute the tangent vector fields
\begin{align*}
\Phi_t=\alpha'\left(-\overline{\alpha},\omega\right),\qquad \Phi_s=\dot\omega\left(\overline{\omega},\alpha\right).
\end{align*}
Then we obtain
$|\Phi_t|^2=|\alpha'|^2\left(|\alpha|^2+|\omega|^2\right)$,
$|\Phi_s|^2=|\dot\omega|^2\left(|\alpha|^2+|\omega|^2\right)$ and
$\left(\Phi_t,\Phi_s\right)=0$. So $\Phi$ is a Lagrangian
immersion whose induced metric is written as in (\ref{metricPhi}).

Taking imaginary parts in $(\Phi_{tt},\Phi_t)$,
$(\Phi_{tt},\Phi_s)$, $(\Phi_{ss},\Phi_t)$ and
$(\Phi_{ss},\Phi_s)$, we obtain the formulas given for the tensor
$C$ in (\ref{TensorC}).

Using the definition of the Lagrangian angle map $\beta $, we obtain that
\[
e^{i\beta}=\det_\c\Bigl(\frac{\Phi_t}{|\Phi_t|},\frac{\Phi_s}{|\Phi_s|}\Bigr)=-\frac{\alpha'\dot\omega}{|\alpha'||\dot\omega|}
\]
and so we arrive at \eqref{LagrangianAngle}.

Finally, we can get the expresion (\ref{MeanCurvatureVector}) for $H$ using (\ref{beta}) taking into account that
$\left(\arg\alpha'\right)'=|\alpha'|\kappa_\alpha$ and $( \arg \dot \omega )\dot\ =|\dot\omega|\kappa_\omega$, or
directly from (\ref{TensorC}) using the orthonormal frame $\Bigl\{
e_1=\frac{\partial_t}{|\alpha'|\sqrt{|\alpha|^2+|\omega|^2}},\, e_2=
\frac{\partial_s}{|\dot\omega|\sqrt{|\alpha|^2+|\omega|^2}} \Bigr\}$.
\end{proof}

\begin{remark}
{\rm Up to a translation, we can rewrite the Lagrangian immersion
$\Phi=\alpha \ast \omega $ as
\begin{equation*}
\Phi(t,s)=\left( \frac{|\omega(s)|^2}{2}+ \int_{s_0}^s \langle
\dot\omega,J\omega\rangle (y)
dy-\frac{|\alpha(t)|^2}{2}-\int_{t_0}^t\ \langle \alpha', J
\alpha\rangle (x)dx \, , \, \alpha(t)\omega(s)\right).
\end{equation*}
From \eqref{metricPhi}, we also observe that $\left(t^*,s^*\right)\in I_1 \times I_2$
is a singular point of $\Phi=\alpha \ast \omega$ if and only
if $\alpha(t^*)=0=\omega(s^*)$.}
\end{remark}

\begin{remark}
{\rm Interchanging the roles of $\alpha $ and
$\omega$ is produced congruent Lagrangians in $\c^2$. In addition,
the same happens with rotations of $\alpha$ and/or $\omega $. But
only if we consider the same homotethy for $\alpha $ and $\omega$
we get homothetic Lagrangian immersions, concretely, $\rho \alpha
\ast \rho \omega = \rho^2 \alpha \ast \omega, \, \forall \rho >0$.
}
\end{remark}
For example, the totally geodesic Lagrangian plane is recovered in
our construction simply considering straight lines passing through
the origin, $\alpha(t)=t$, $t\in\R$, $\omega(s)=s$, $s\in\R$,
since in this case
\begin{align}\label{totallyGeodesic}
\left(\alpha\ast\omega\right)(t,s)=\left(\frac{s^2-t^2}{2},t\,s\right).
\end{align}

\vspace{0.5cm}


\section{Applications}

This section is devoted to study several families of
Lagrangian surfaces in our construction described in Theorem
\ref{thm_immersion}; those characterized by different geometric
properties related with the behavior of the mean curvature vector.

\subsection{Special Lagrangians}
A Lagrangian oriented surface is called \emph{special} if its Lagrangian angle is constant ($\beta\equiv\beta_0$).
From \eqref{beta} this means that $H=0$, that is, the Lagrangian immersion is minimal,
but in fact is area-minimizing because they are calibrated by ${\rm Re} (e^{-i\beta_0}\Omega)$ (see \cite{HL}).
It is well known that these surfaces should be holomorphic curves with respect to another complex structure on $\C^2$.

\begin{corollary}\label{cor_SpecialLagrangians}
The immersion $\Phi=\alpha \ast \omega $, given in Theorem \ref{thm_immersion}, is special Lagrangian if and only if
$\alpha $ and $\omega$ are straight lines.

Putting $\alpha (t)=t+i\,a$, $a\in \R$, and $\omega (s)=s+i\,b$, $b\in
\R$, then $\Phi=\alpha\ast\omega$ can be written, up to a translation in $\C^2$, by
\begin{align}\label{minimalPhi}
\Phi(t,s)=\left( \frac{s^2-t^2}{2}+i(at-bs),  ts+i(as+bt) \right),
\ (t,s)\in\R^2.
\end{align}
If $(a,b)=(0,0)$, we get the totally geodesic Lagrangian plane
\eqref{totallyGeodesic}. If $(a,b)\neq (0,0)$, these
special Lagrangians correspond to the holomorphic curves $z
\mapsto c\,z^2 $, $c=-\frac{(a-ib)^2}{2(a^2+b^2)^2}\in \c^*$.
\end{corollary}

\begin{remark}
{\rm According to \cite{HO}, the graphs $z \mapsto c\,z^2, \,
c>0$, are the only complete orientable minimal surfaces in
Euclidean $n$-space with finite total curvature $-2\pi$.
Topologically they are cylinders lying in $\R^4$ and
\eqref{minimalPhi} provides a conformal Lagrangian parametrization
of them. }
\end{remark}

\begin{proof}
From \eqref{MeanCurvatureVector} we get that the minimality of
$\Phi=\alpha\ast\omega$ is equivalent to
$\kappa_\alpha=0=\kappa_\omega$. Then
 $\alpha$ and $\omega$ must
be straight lines. So, up to rotations, we can consider
$\alpha(t)=t+i\,a$, $a\in\R$, and $\omega(s)=s+i\,b$, $b\in\R$,
and it is a straightforward computation to get \eqref{minimalPhi}.

If $(a,b)\neq(0,0)$, considering $\varphi:\C^2\rightarrow\C^2$,
$\varphi(x_1+i\,y_1,x_2+i\,y_2)=(y_1+i\,y_2,x_1-i\,x_2)$ and $w=t+i\,s\in\C$, one easily obtains the holomorphic map
\begin{align*}
\left(\varphi\circ\Phi\right)(w)=\left((a+i\,b)\,w,-\frac{w^2}{2}\right)\subset\C^2.
\end{align*}
Letting $w=(a+i\,b)\,z$, we finally get $\left(\varphi\circ\Phi\right)(z)=\Bigl(z,-\frac{(a-ib)^2}{2(a^2+b^2)^2}\,z^2\Bigr)$ and the proof is finished.
\end{proof}

\vspace{0.3cm}

\subsection{Lagrangians with parallel mean curvature vector}
From the point of view of the mean curvature vector $H$, the easiest (non minimal) examples are those with parallel
mean curvature vector, i.e., $\nabla^\perp H=0$, where $\nabla^\perp$ is the connection in the normal bundle. In the
Lagrangian setting, the complex structure $J$ defines an isomorphism between the tangent and the normal bundles, so the
condition to have parallel mean curvature vector is equivalent to the fact that $JH$ is a parallel vector field.

In the next result, we show that the right  circular cylinder is the only Lagrangian with parallel mean curvature vector in our construction.
\begin{corollary}\label{cor_PMC}
The Lagrangian immersion $\Phi=\alpha \ast \omega $, given in Theorem \ref{thm_immersion}, has parallel (non null) mean
curvature vector if and only if $\alpha $ and $\omega$ are both circles centered at the origin.

Putting $\alpha (t)=e^{it}$ and $\omega (s)=R\,e^{is}$, $R>0$, then
$\Phi=\alpha\ast\omega$ can be written by
\begin{align}\label{PMCPhi}
\Phi(t,s)=\left( i(R^2 s -t),R\,e^{i(s+t)} \right), \ (t,s)\in\R^2
\end{align}
that describes the right circular cylinder $\R \times \s^1 (R)$.
\end{corollary}

\begin{proof}
Without loss of generality we can consider $\alpha$ and $\omega$ curves parametrized  by arclength. From
\eqref{metricPhi} the induced metric is given by $e^{2u}\left(dt^2+ds^2\right)$, where $e^{2u}=|\alpha|^2+|\omega|^2$.
Using \eqref{MeanCurvatureVector}, it is not difficult to check that
$JH=-e^{-2u}(\kappa_\alpha\partial_t+\kappa_\omega\partial_s)$ is a parallel vector field if and only if
\begin{align}\label{curvatura-paralela}
\kappa_\alpha'-u_t\kappa_\alpha+u_s\kappa_\omega=0,\quad u_t\kappa_\omega+u_s\kappa_\alpha=0,\quad \dot\kappa_\omega-u_s\kappa_\omega+u_t\kappa_\alpha=0.
\end{align}
From the first and the third equation of
\eqref{curvatura-paralela} we deduce that
$\kappa_\alpha'+\dot\kappa_\omega=0$ and hence
$\kappa_\alpha(t)=a\,t+b$ and $\kappa_\omega(s)=-a\,s+c$, with
$a$, $b$, $c\in\R$. We distinguish three cases: We first suppose
that $u_s=0$. It is equivalent to $|\omega|$ is constant. Using
\eqref{curvatura-paralela} and the fact that $\phi$ is non
minimal, it follows that $u_t=0$ what means that $|\alpha|$ is
also constant. If $u_t=0$ similarly we get that $|\alpha|$ and
$|\omega|$ are constants. Finally, if $u_t\neq0$ and $u_s\neq0$,
from the second equation of \eqref{curvatura-paralela}, there
exists $c_1\in\R^*$ such that $\kappa_\alpha=c_1u_t$ and
$\kappa_\omega=-c_1u_s$. So $u_t=(at+b)/c_1$ and $u_s=(as-c)/c_1$.
Putting this information in \eqref{curvatura-paralela}, we arrive
at $a=b=c=0$ which is a contradiction since we are assuming that
$H\neq 0$.
\end{proof}

\vspace{0.3cm}

\subsection{Hamiltonian stationary Lagrangians}
A Lagrangian surface is called {\em Hamiltonian stationary} if the
Lagrangian angle $\beta $ is harmonic, i.e. $\Delta \beta =0$,
where $\Delta $ is the Laplace operator on $M$. Using
\eqref{beta}, this is equivalent to the vanishing of the
divergence of the tangent vector field $JH$. Hamiltonian
stationary Lagrangian (in short HSL) surfaces are critical points
of the area functional with respect to a special class of
infinitesimal variations preserving the Lagrangian constraint;
namely, the class of compactly supported Hamiltonian vector fields
(see \cite{O}). Special Lagrangians and Lagrangians with parallel
mean curvature vector are trivial examples of HSL surfaces in
$\c^2$; more interesting examples can be found in \cite{A}
\cite{AC}, \cite{CU2} and \cite{HR1}.
\begin{corollary}\label{cor_HSL}
The immersion $\Phi=\alpha \ast \omega $, given in Theorem \ref{thm_immersion}, is Hamiltonian stationary Lagrangian if
and only if the curvature functions $k_\alpha$ and $k_\omega$ of $\alpha$ and $\omega$ are given by $k_\alpha (t)=at+b$
and $k_\omega (s)= -as+c$, with $a,b,c\in \R$,  and where $t$ and $s$ are the arclength parameters of $\alpha$ and
$\omega$, respectively.
\end{corollary}

\begin{proof}
Without loss of generality we can consider $\alpha$ and $\omega$ parametrized  by arclength and, in this way $\Phi$ is
conformal. The Lagrangian surface $\Phi=\alpha\ast\omega$ is Hamiltonian stationary if and only if the Lagrangian angle
map $\beta $ verifies  $\Delta\beta=0$. So, using \eqref{LagrangianAngle}, we get that $\Phi$ is HSL if and only if the
curvature functions $\kappa_\alpha$ and $\kappa_\omega$ of $\alpha$ and $\omega$ satisfy
$\kappa_\alpha'+\dot\kappa_\omega=0$, where $'$ and $^\cdot$ denote the derivatives respect to  the arclength
parameters $t$ and $s$, respectively. Then there exists $a\in\R$ such that $\kappa_\alpha'=a=-\dot\kappa_\omega$ and
this finishes the proof of the corollary.
\end{proof}

We distinguish two essential cases in the family described in
Corollary \ref{cor_HSL}:

\vspace{0.1cm}

\emph{Case 1}: $a=0$, i.e.\ $\kappa_\alpha \equiv b$ and
$\kappa_\omega \equiv c$. So, $\alpha$ and $\omega$ are either
straight lines or circles. If $\alpha$ and $\omega$ are both
straight lines we lie in the situation of Corollary
\ref{cor_SpecialLagrangians} obtaining the special Lagrangians
described there. Otherwise, we get the following subcases: either
$b\neq0$ and $c=0$ (or $b=0$ and $c\neq0$) or $b\neq0$ and
$c\neq0$.

First, if $b\neq0$ and $c=0$, taking $\alpha(t)=a_0+R\,e^{it/R}$,
with $R=1/|b|>0$, and $\omega(s)=s+i\,b_0$, $a_0,b_0\geq0$, then
$\Phi=\alpha \ast \omega$ can be written, up to a translation in
$\C^2$, by
\begin{align*}
\Phi(t,s)=\left(\frac{s^2}{2}-R\,a_0\,e^{it/R}-i(b_0\,s+Rt),a_0s+R(s+ib_0)e^{it/R}\right),\,
(t,s)\in\R^2.
\end{align*}
In particular, when $a_0=b_0=0$, $R=1$ we get $(t,s)\mapsto\left(\frac{s^2}{2}-it,s\,e^{it}\right)$, which corresponds
to the complete non-trivial HSL plane described in Corollary 3.5 of \cite{CL12}.

Second, if $b\neq0$ and $c\neq0$, up to a dilation we can consider
$b=1$, $|c|=1/R$, and take $\alpha (t)=a_1+e^{it}$ and $\omega
(s)=a_2+R e^{is/R}$, $a_1,a_2\geq 0, \, R>0$. Then $\Phi=\alpha
\ast \omega$ can be written, up to a translation in $\C^2$, by
\begin{align*}
\Phi(t,s)=\left( a_2 R e^{is/R}\!-\! a_1 e^{it}\!+\!i(R s -t), a_1
R e^{is/R}\!+\! a_2 e^{it}\!+ \!R\,e^{i(t+s/R)} \right),
\,(t,s)\in \R^2.
\end{align*}
In particular, when $a_1=a_2=0$ we recover the right  circular cylinder $\R\times\mathbb{S}^1(R)$.

The above immersions provide conformal parametrizations of HSL
complete surfaces, some of them studied in \cite{A} from a different approach.

\vspace{0.1cm}

\emph{Case 2}: $a\neq0$, i.e.\ $\kappa_\alpha$ and $\kappa_\omega$
are certain linear functions of the arc parameter. After applying
suitable translations on the parameter, we can consider
$\kappa_\alpha(t)=a\,t$ and $\kappa_\omega(s)=-a\,s$. Thus, in
this case $\alpha$ and $\omega$ must be Cornu spirals with
opposite parameter. The corresponding immersions
$\alpha\ast\omega$ provide new examples of HSL surfaces.

\vspace{0.3cm}

\subsection{Lagrangians with constant mean curvature}
A Lagrangian surface has constant mean curvature if $|H|$ is
constant. Examples of Lagrangians with constant mean curvature can
be found in \cite{CU4}
\begin{corollary}\label{cor_Hconstant}
The immersion $\Phi=\alpha \ast \omega $, given in Theorem \ref{thm_immersion}, has constant mean curvature
$|H|\equiv\rho>0$ if and only if the curvature functions $k_\alpha$ and $k_\omega$ of $\alpha$ and $\omega$ satisfy,
respectively, $\kappa_\alpha^2=\rho^2|\alpha|^2-\lambda$ and $\kappa_\omega^2=\rho^2|\omega|^2+\lambda$, for some
$\lambda\in\R$.
\end{corollary}

\begin{proof}
Using the expresion \eqref{MeanCurvatureVector}, it follows that
\begin{align*}
|H|^2=\frac{1}{|\alpha|^2+|\omega|^2}\left(\kappa_\alpha^2+\kappa_\omega^2\right).
\end{align*}
If $|H|\equiv\rho$, since $\alpha$ depends on $t$ and $\omega$
depends on $s$, there exists $\lambda\in\R$ such that
$\kappa_\omega^2-\rho^2|\omega|^2=\lambda=\rho^2|\alpha|^2-\kappa_\alpha^2$,
what proves the result.
\end{proof}

Now we show how the conditions on $\alpha$ and $\omega$ given in
Corollary \ref{cor_Hconstant} determine both curves. If we take
$\alpha$ and $\omega$ planar curves parametrized  by arclength,
they can be written as follows
\begin{align*}
\alpha &=r_1\,e^{i\int\frac{\sqrt{1-r_1'^2}}{r_1}},\quad
r_1=r_1(t)=|\alpha(t)|,\\
\omega &=r_2\,e^{i\int\frac{\sqrt{1-\dot r_2^2}}{r_2}},\quad r_2=r_2(s)=|\omega(s)|.
\end{align*}
It is not difficult to check that the curvatures $\kappa_\alpha$ and $\kappa_\omega$ can be expressed in terms of the
derivatives of $r_1$ and $r_2$ by the following equations:
\begin{align*}
r_1\sqrt{1-r_1'^2}\,\kappa_\alpha=1-r_1'^2-r_1r_1'',\qquad r_2\sqrt{1-\dot r_2^2}\,\kappa_\omega=1-\dot r_2^2-r_2\ddot r_2.
\end{align*}
Studying the case of Corollary \ref{cor_Hconstant}, i.e.
$\kappa_\alpha^2=\rho^2r_1^2-\lambda$  and
$\kappa_\omega^2=\rho^2r_2^2+\lambda$, we get the following
ordinary differential equations for $r_1$ and $r_2$:
\begin{align}\label{edo-r1}
(1-r_1'^2-r_1r_1'')^2=\left(\rho^2 r_1^2-\lambda\right)r_1^2(1-r_1'^2),
\end{align}
\begin{align}\label{edo-r2}
(1-\dot r_2^2-r_2\ddot r_2)^2=\left(\rho^2 r_2^2+\lambda\right)r_2^2(1-\dot r_2^2).
\end{align}
When we consider $r_1$ and $r_2$ constant, we recover the right
circular cylinder obtained in the Corollary \ref{cor_PMC}
corresponding to the parallel mean curvature case.

In the general case, we are able to obtain first integrals of the differential equations \eqref{edo-r1} and \eqref{edo-r2}:
\begin{align}\label{eq-rt}
\frac{\left(\rho^2 r_1^2-\lambda\right)^{3/2}}{\rho^2}+\mu_1=3\,r_1\sqrt{1-r_1'^2},
\end{align}
and
\begin{align}\label{eq-rs}
\frac{\left(\rho^2 r_2^2+\lambda\right)^{3/2}}{\rho^2}+\mu_2=3\,r_2\sqrt{1-\dot r_2^2},
\end{align}
where $\mu_1$ and $\mu_2$ are arbitrary constants.

This shows that the family of Lagrangian surfaces with constant
mean curvature $\rho>0$ in our construction with planar curves is
very large. In general, the solutions of \eqref{eq-rt} and
\eqref{eq-rs} are not easy to control, appearing hyperelliptic
functions in most cases. We finish this section considering the
following illustrative situation:

Let $\lambda=\mu_1=\mu_2=0$. Up to dilations, we can suppose $\rho=3$. Then equations \eqref{eq-rt} and \eqref{eq-rs}
coincide and they are reduced to the differential equation
\begin{equation}\label{eq-r3}
r'^2+r^4=1
\end{equation}
Using (\ref{eq-r3}) we get that the generatrix curves are given by
\begin{equation}\label{eq-curvas3}
\alpha (t) = r(t) e^{i\int r(t)dt}, \quad \omega (s) = r(s) e^{i\int r(s)ds}
\end{equation}
and so, taking into account (\ref{eq-r3}) again, we arrive at
\begin{align}\label{eq-surface3}
\ (\alpha \ast \omega)(t,s)=\Biggl( \frac{r(s)^2-r(t)^2}{2} + i \left( \int r(s)^3 ds - \int r(t)^3 dt \right)
,
\\
 r(t)r(s) e^{i (\int r(t)dt + \int r(s)ds)} \Biggr) \nonumber
\end{align}
The solution of (\ref{eq-r3}) can be expressed in terms of some elliptic functions. Concretely, from formulas 160.01
and 318.01 of \cite{BF}, we can write that
\[
r(t)=\text{sn}(t,i)=\frac{1}{\sqrt 2}\frac{\text{sn}(\sqrt 2 \,t,1/\sqrt 2)}{\text{dn}(\sqrt 2 \,t,1/\sqrt 2)},
\]
and
\[
\theta(t):= \int r(t)dt = -\arctan \left( \frac{1+\text{cn}(\sqrt 2 \,t,1/\sqrt 2)}{1-\text{cn}(\sqrt 2 \,t,1/\sqrt 2)}
\right),
\]
where $\text{sn}$, $\text{cn}$ and $\text{dn}$ are the elementary Jacobi elliptic functions usually known as sine
amplitude, cosine amplitude and delta amplitude respectively (see \cite{BF} for background in Jacobi elliptic
functions). Then it is not difficult to get that $r=\sin 2 \theta$ and hence $\alpha $ and $\omega $ are both
Bernoulli's lemniscatae (see Figure \ref{Lemniscata}). Finally, using formula 318.03 of \cite{BF}, we can conclude from
(\ref{eq-surface3}) that $\Phi=\alpha \ast \omega$ is given explicitly by the following:
\begin{align*}
\Phi (t,s)= \frac{1}{4} \biggl( \text{sd}^2(\sqrt 2 \,s)-\text{sd}^2(\sqrt 2 \,t)+i( \text{cd}(\sqrt 2
\,t)\text{nd}(\sqrt 2 \,t)-
\text{cd}(\sqrt 2 \,s) \text{nd}(\sqrt 2 \,s)) , \\
 2 \text{sd}(\sqrt 2 \,t)\text{sd}(\sqrt 2 \,s) \exp \left( -i\left( \arctan \left( \tfrac{1+\text{cn}(\sqrt
2\,t)}{1-\text{cn}(\sqrt 2 \,t)} \right)+\arctan \left( \tfrac{1+\text{cn}(\sqrt 2\,s)}{1-\text{cn}(\sqrt 2 \,s)}
\right) \right) \right) \biggr),
\end{align*}
\begin{figure}[h!]
\begin{center}
\includegraphics[height=3cm]{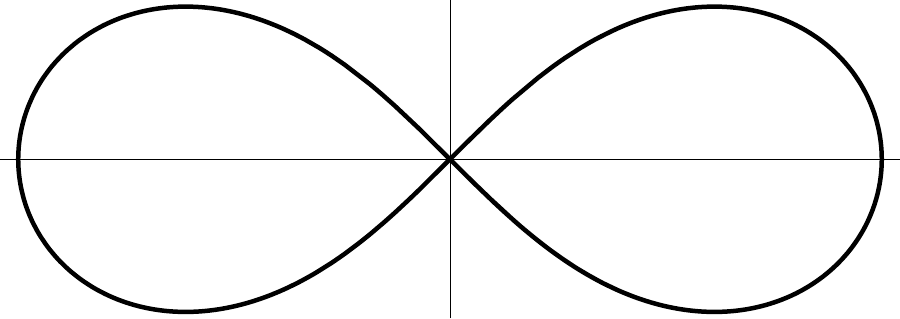}\caption{Bernoulli's lemniscata.}\label{Lemniscata}
\end{center}
\end{figure}
where $\text{sd}u=\frac{\text{sn}u}{\text{dn}u}$, $\text{cd}u=\frac{\text{cn}u}{\text{dn}u}$ and $\text{nd}u=\frac{1}{\text{dn}u}$.
Since it is doubly-periodic, this provides a (branched) Lagrangian torus with constant mean curvature vector $|H|^2=9$. In Figure \ref{Hcte} we illustrate the projections of $\Phi$ to the coordinate 3-spaces of $\R^4$.

\begin{center}
\begin{figure}[h]
\begin{center}
\begin{tabular}{ccc}
\includegraphics[height=4.5cm]{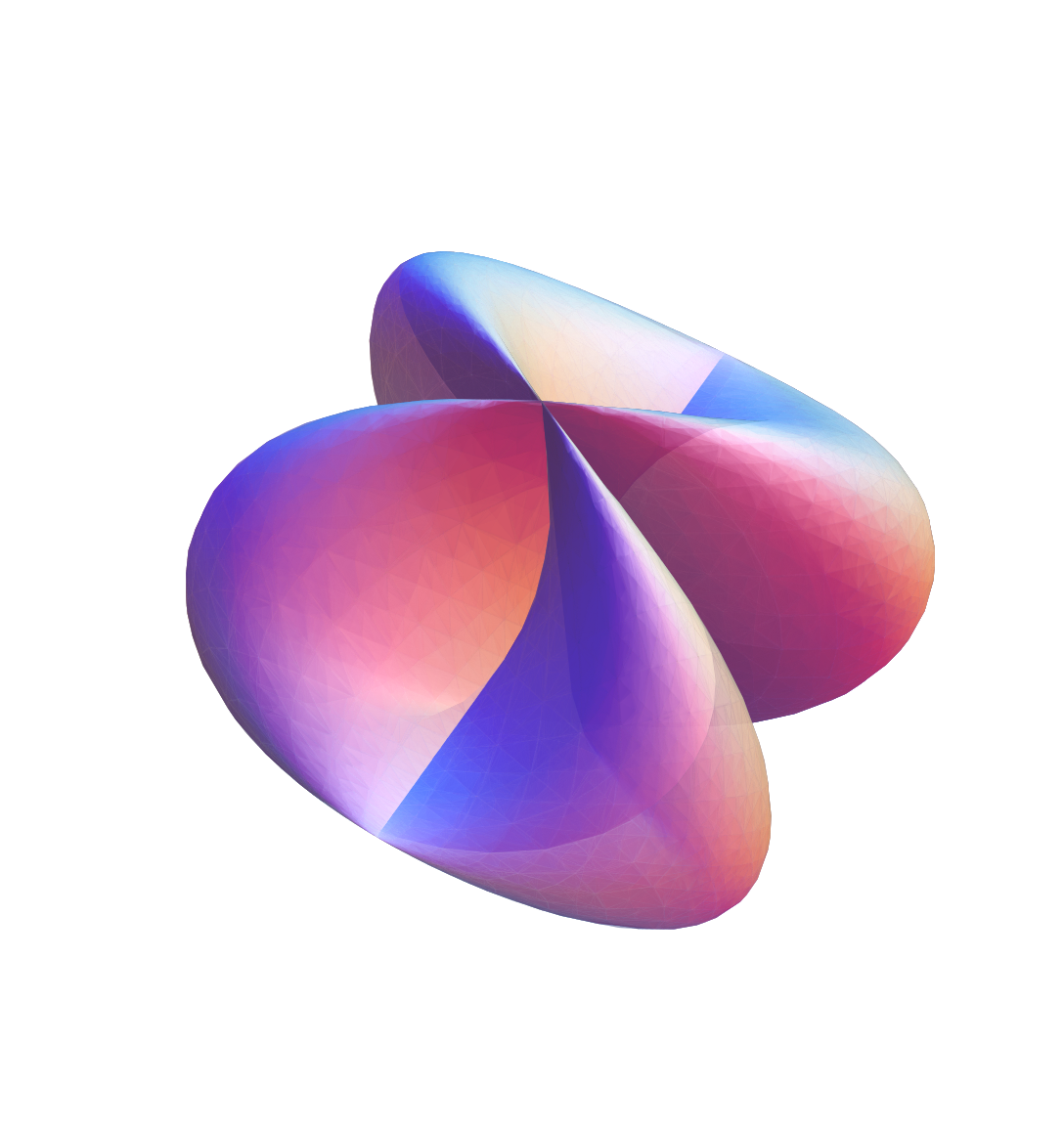} &\hspace{1cm} & \includegraphics[height=4.5cm]{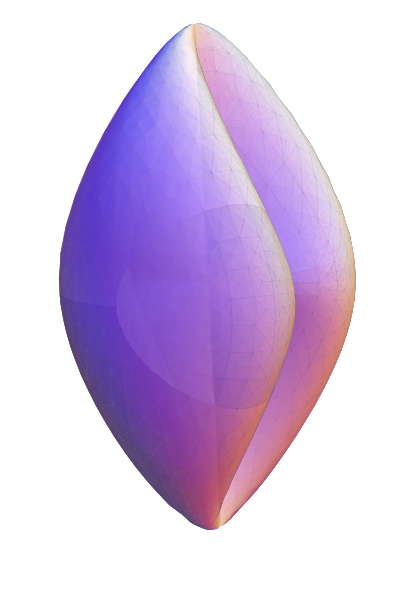}\\
\includegraphics[height=4.5cm]{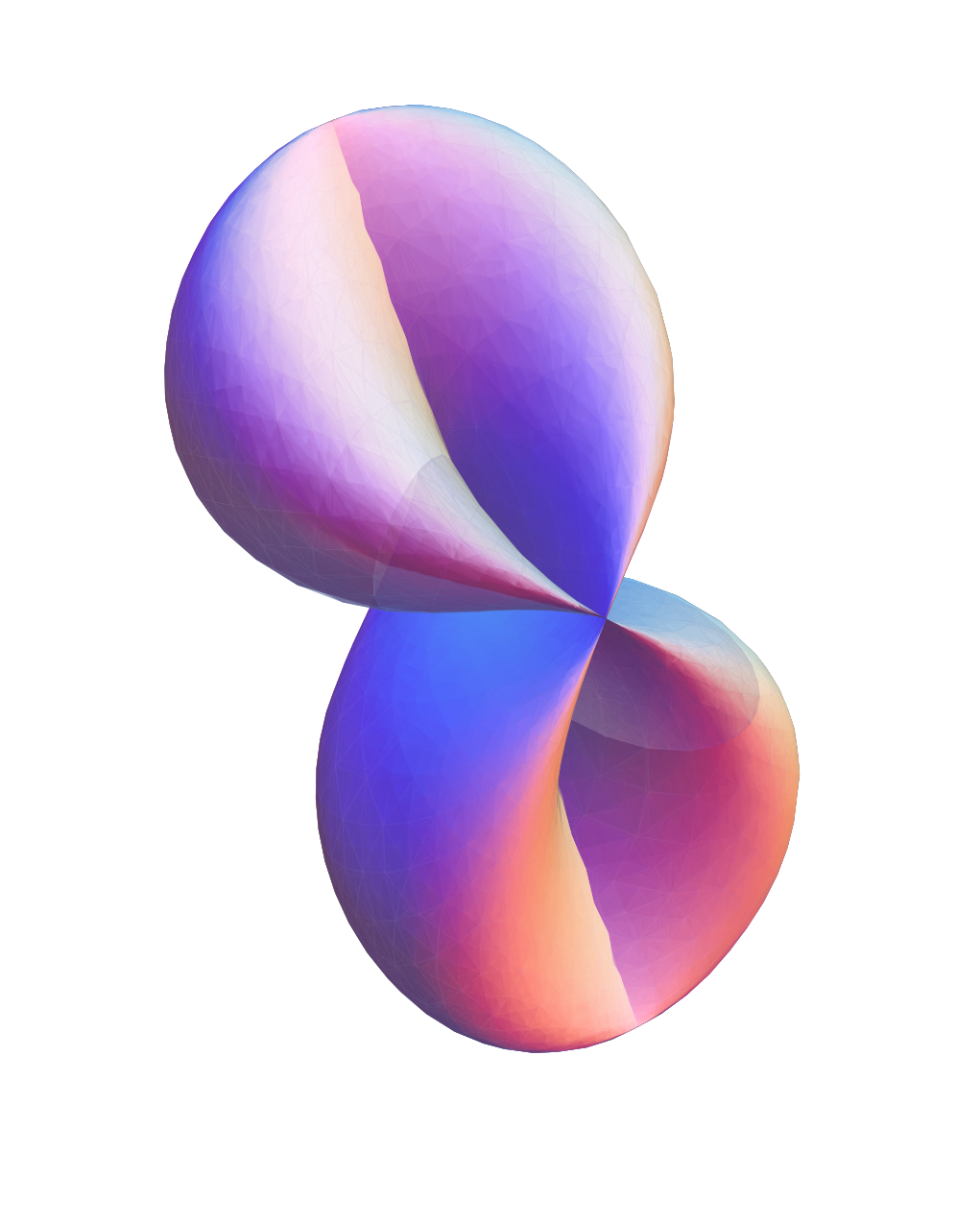} &\hspace{1cm} & \includegraphics[height=4.5cm]{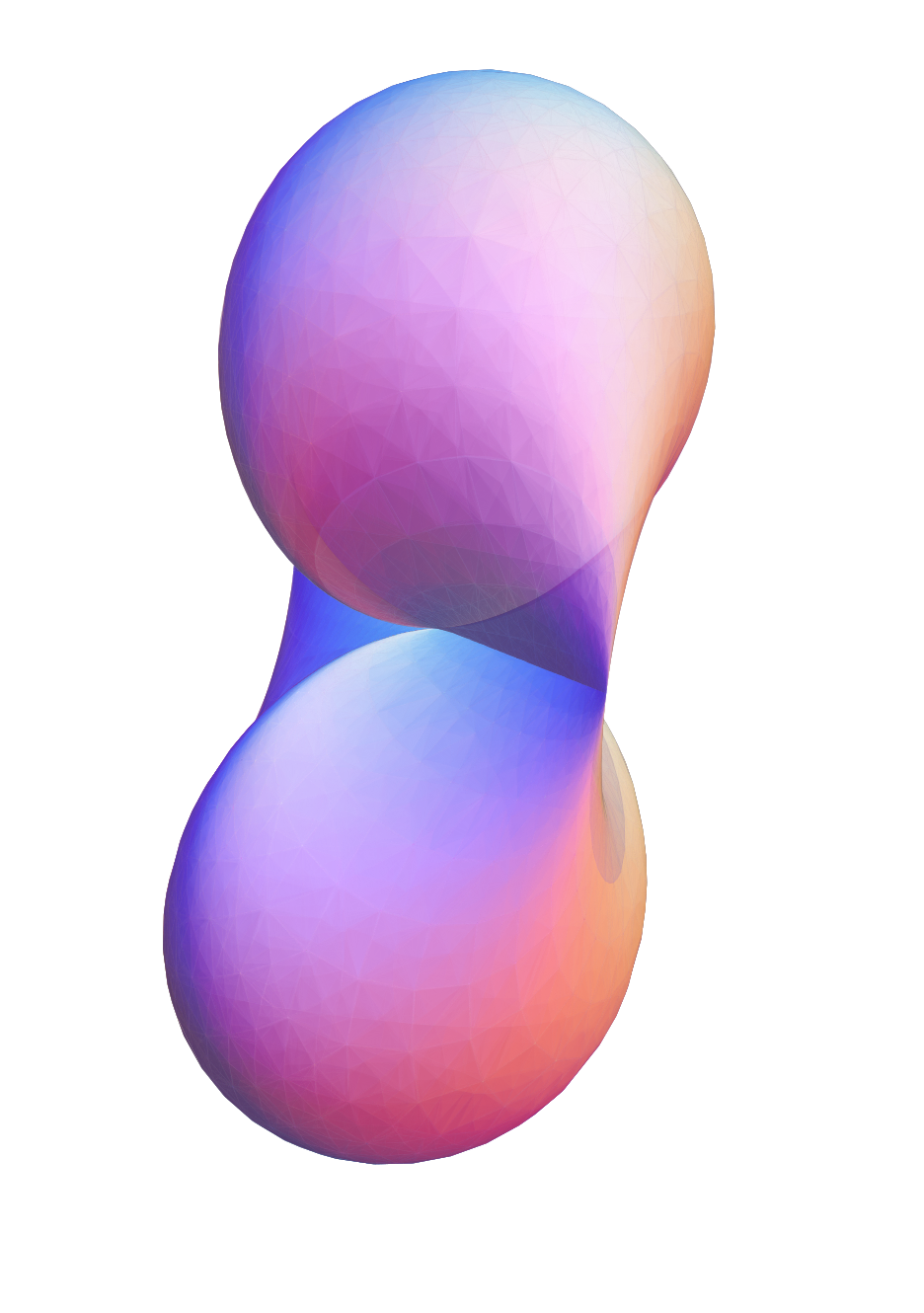}\\
\end{tabular}
\caption{Projections of the Lagrangian torus with $|H|^2=9$ to the coordinate 3-spaces of
$\C^2\equiv\R^4$.}\label{Hcte}
\end{center}
\end{figure}
\end{center}
\vspace{0.3cm}

\subsection{Lagrangian self-similar solitons}
An immersion $\phi: M\rightarrow \r^4$ is called a self-similar
solution for mean curvature flow if
\begin{equation}\label{self}
H=\pm\phi^\perp
\end{equation}
where $\phi^\perp$ denotes the normal projection of the position
vector $\phi$. If
$H=-\phi^\perp$ it is called a self-shrinker, and if
$H=\phi^\perp$ it is called a self-expander. Examples of Lagrangian self-shrinkers and self-expanders can be found in \cite{CL10}.

In the next result we show that the cylinder $\mathbb{S}^1\times\R$ is the only (non totally geodesic) self-shrinker
for the mean curvature flow in our construction.

\begin{corollary}\label{cor_self-similar}
The Lagrangian immersion $\Phi=\alpha \ast \omega$, given in Theorem \ref{thm_immersion}, is a (non totally geodesic)
self-similar solution for mean curvature flow if and only if $\alpha $ and $\omega$ are both circles of radius one
centered at the origin, so that $\Phi=\alpha\ast\omega$ describes the right circular cylinder $\R \times \s^1$.
\end{corollary}

\begin{proof}
We first calculate $\Phi^\perp$ as follows:
\begin{align*}
\Phi^\perp=&\frac{1}{|\alpha|^2+|\omega|^2}\left(\frac{\im\left(\Phi,\Phi_t\right)}{|\alpha'|^2}\,J\Phi_t+\frac{\im\left(\Phi,\Phi_s\right)}{|\dot\omega|^2}\,J\Phi_s\right)\\
=&\frac{1}{|\alpha'|^2(|\alpha|^2+|\omega|^2)}\Biggl(\frac{|\alpha|^2+|\omega|^2}{2}\langle\alpha,J\alpha'\rangle\\
&+\left(\int\langle\alpha',J\alpha\rangle-\int\langle\dot\omega,J\omega\rangle\right)\langle\alpha,\alpha'\rangle\Biggr)J\Phi_t\\
&+\frac{1}{|\dot\omega|^2(|\alpha|^2+|\omega|^2)}\Biggl(\frac{|\alpha|^2+|\omega|^2}{2}\langle\omega,J\dot\omega\rangle\\
&+\left(\int\langle\dot\omega,J\omega\rangle-\int\langle\alpha',J\alpha\rangle\right)\langle\omega,\dot\omega\rangle\Biggr)J\Phi_s.
\end{align*}
Taking into account \eqref{MeanCurvatureVector}, we have that
$H=\epsilon\,\Phi^\perp$, $\epsilon=\pm1$, if and only if  the curvatures of the
curves $\alpha$ and $\omega$ satisfy
\begin{align*}
\kappa_\alpha=&\epsilon\left(\frac{|\alpha|^2+|\omega|^2}{2}\langle\alpha,J\alpha'\rangle+\left(\int\langle\alpha',J\alpha\rangle-\int\langle\dot\omega,J\omega\rangle\right)\langle\alpha,\alpha'\rangle\right),\\
\kappa_\omega=&\epsilon\left(\frac{|\alpha|^2+|\omega|^2}{2}\langle\omega,J\dot\omega\rangle+\left(\int\langle\dot\omega,J\omega\rangle-\int\langle\alpha',J\alpha\rangle\right)\langle\omega,\dot\omega\rangle\right).
\end{align*}
Now we derivate the above equations with respect to $s$ and $t$ respectively and we
obtain the following necessary condition
\begin{align}\label{condi-kappa}
\langle\omega,\dot\omega\rangle\langle\alpha,J\alpha'\rangle=\langle\dot\omega,J\omega\rangle\langle\alpha,\alpha'\rangle.
\end{align}
Using \eqref{condi-kappa} together with the conditions on
$\kappa_\alpha$ and $\kappa_\omega$, we get that necessarily $\epsilon=-1$ and so $\Phi$ is a
self-shrinker and  the only possibility is that the curves $\alpha$ and $\omega$ are both
circles centered at the origin. Corollary \ref{cor_PMC} finishes the proof.
\end{proof}

\vspace{0.3cm}

\subsection{Lagrangian translating solitons}
An immersion $\phi: M\rightarrow \r^4$ is called a translating
soliton for mean curvature flow if
\begin{equation}\label{trl}
H={\bf e}^\perp
\end{equation}
for some nonzero constant vector ${\bf e}\in \R^4$, where ${\bf
e}^\perp $ denotes the normal projection of the vector $\bf e $, which can be fixed up to congruences. Examples of Lagrangian translating solitons can be found in \cite{CL12}.

\begin{corollary}\label{cor_translating}
The Lagrangian immersion $\Phi=\alpha \ast \omega $, given in Theorem \ref{thm_immersion}, is a translating soliton
with translating vector $\mathbf{e}=\left(\rho\,e^{i\theta},0\right)\in\c^2$ if and only if the planar curves $\alpha$
and $\omega$ satisfy that
\begin{equation}\label{curvaturas_translating}
|\alpha'|\kappa_\alpha=\rho\,\im\left(e^{-i\,\theta}\alpha'\overline\alpha\right),\qquad
|\dot\omega|\kappa_\omega=-\rho\,\im\left(e^{-i\,\theta}\dot\omega\overline\omega\right).
\end{equation}
\end{corollary}
\begin{remark}
{\rm In Corollary \ref{cor_translating} we recover, up to dilations and isometries, the Lagrangian translating solitons
described in Proposition 3.3 of \cite{CL12}. The corresponding curves $\alpha$ and $\omega$ are special non trivial
solution of the curve shortening problem including spirals and self-shrinking and self-expanding planar curves (see
\cite{CL12} and references therein). }
\end{remark}
\begin{proof}
We first compute $\left(\rho\,e^{i\theta},0\right)^\perp$ for $\Phi=\alpha\ast\omega$:
\begin{align*}
\left(\rho\,e^{i\theta},0\right)^\perp&=\frac{\rho}{|\alpha|^2+|\omega|^2}\left(\frac{\im\left(\left(e^{i\theta},0\right),\Phi_t\right)}{|\alpha'|^2}\,J\Phi_t+\frac{\im\left(\left(e^{i\theta},0\right),\Phi_s\right)}{|\dot\omega|^2}\,J\Phi_s\right)\\
&=\frac{\rho}{|\alpha|^2+|\omega|^2}\left(\frac{\im\left(e^{-i\theta}\alpha'\overline\alpha\right)}{|\alpha'|^2}J\Phi_t-\frac{\im\left(e^{-i\theta}\dot\omega\overline\omega\right)}{|\dot\omega|^2}J\Phi_s\right).
\end{align*}
Then, taking into account \eqref{MeanCurvatureVector}, we finish the proof.
\end{proof}

\vspace{0.3cm}


\subsection{Willmore Lagrangians}
Consider the Willmore functional
\begin{align*}
\mathcal{W}=\int_\Sigma|H|^2d\mu
\end{align*}
for a closed surface $\Sigma$ immersed  in Euclidean space. The critical points of $\mathcal{W}$ are known as Willmore surfaces. Examples of Lagrangian Willmore surfaces can be found in \cite{CU3}.

We take $\alpha$ and $\omega$ closed unit speed
planar curves. Using Theorem  \ref{thm_immersion}, the Willmore functional of the Lagrangian conformal immersion $\Phi=\alpha\ast\omega$ is given by
\begin{align}\label{willmore_phi}
\mathcal{W}_\Phi=\int_{I_1\times I_2}\left(\kappa_\alpha^2+\kappa_\omega^2\right)dtds=L(\omega)\int_{I_1}\kappa_\alpha^2dt+L(\alpha)\int_{I_2}\kappa_\omega^2ds,
\end{align}
where $L(\alpha)$ and $L(\omega)$ denote the lengths of $\alpha$ and $\omega$, respectively.

\begin{corollary}\label{cor_willmore}
The Lagrangian immersion $\Phi=\alpha \ast \omega $, given in Theorem \ref{thm_immersion}, is a critical point of the
Willmore functional $\mathcal{W}_\Phi$ (with fixed lengths $L(\alpha)$ and $L(\omega)$) if and only if the curves
$\alpha$ and $\omega$ are free elastic curves parametrized by the arc length.
\end{corollary}
\begin{proof}
From \eqref{willmore_phi}, the critical points of the Willmore functional $\mathcal{W}_\Phi$ (with fixed
$L_1=L(\alpha)$ and $L_2=L(\omega)$) are given by Lagrangian conformal immersions constructed with unit speed planar
curves that are critical points of the functionals $\int_0^{L_1}\kappa_\alpha^2dt$ and $\int_0^{L_2}\kappa_\omega^2ds$,
respectively. Since these are precisely free elastic curves according to \cite{LS} we finish the proof.
\end{proof}

\vspace{0.3cm}

\subsection{Lagrangian tori}
We now ask about the possibility of obtaining compact Lagrangians from our construction of Theorem \ref{thm_immersion}.
The following result gives a sufficient condition on the generatrix closed curves to produce Lagrangian tori.
\begin{proposition}\label{cor_compact}
Let $\alpha=\alpha(t)\subset\c\setminus\{0\}$, $t\in\R$, and $\omega=\omega(s)\subset\c\setminus\{0\}$, $s\in\R$, be
regular periodic planar curves, with periods $T$ and  $S$ respectively, such that
\begin{align}\label{condi_curves-compact}
\int_0^{T}\langle\alpha',J\alpha\rangle=0=\int_0^{S}\langle\dot\omega,J\omega\rangle.
\end{align}
Then the Lagrangian immersion $\Phi=\alpha \ast \omega $, given in Theorem \ref{thm_immersion}, is doubly periodic;
concretely, $\Phi(t+T,s)=\Phi(t,s)=\Phi(t,s+S)$, $\forall(t,s)\in\R^2$.
\end{proposition}
\begin{proof}
Under the hypothesis of this proposition, using Remark 2.1, it is clear that $\Phi(t+T,s)=\Phi(t,s)=\Phi(t,s+S)$ if and
only if
\[
\int_{0}^{t+T}\langle\alpha',J\alpha \rangle = \int_{0}^{t}\langle\alpha',J\alpha \rangle , \quad
\int_{0}^{s+S}\langle\dot \omega ,J\omega \rangle = \int_{0}^{s}\langle\dot \omega ,J\omega \rangle .
\]
But using again that $\alpha$ is $T$-periodic and $\omega$ is $S$-periodic, the above conditions are reduced to
(\ref{condi_curves-compact}).
\end{proof}
For example, we can consider a Gerono's lemniscata (see Figure \ref{gerono}) given by
\begin{align*}
\alpha(t)=\left(1 + 2 \cos t, 2 \cos t \sin t\right)
\end{align*}
and a Lissajous curve (see Figure \ref{lissajous}) given by
\begin{align*}
\omega(s)=\left(\sin s, \sin(2 s)\right).
\end{align*}
Then it is easy to check that $\Phi=\alpha\ast\omega$ can be written as
\begin{align*}
\Phi(t,s)=&\biggl(
\frac{1}{4}\bigl(8 \sin^4t-2\cos^2s - \cos(4 s) - 8 \cos t \bigr)\\
&+
 \frac{i}{6}\,\bigl(9 \cos s - \cos(3 s) - 2 \left(9 \sin t + 3 \sin(2 t) + \sin(3 t)\right)\bigr), \\
& \left(1 + \left(2 - 4 \cos s \sin t\right)\cos t\right)\sin s+i\bigl( \left(1 + 2 \cos t\right) \sin(2 s) +
  \sin s \sin(2 t)\bigr)\biggr).
\end{align*}

In Figure \ref{compact} we illustrate the projections of $\Phi=\alpha\ast\omega$ to the coordinate 3-spaces of $\R^4$.

\begin{figure}[h!]
\begin{center}
\includegraphics[height=3.9cm]{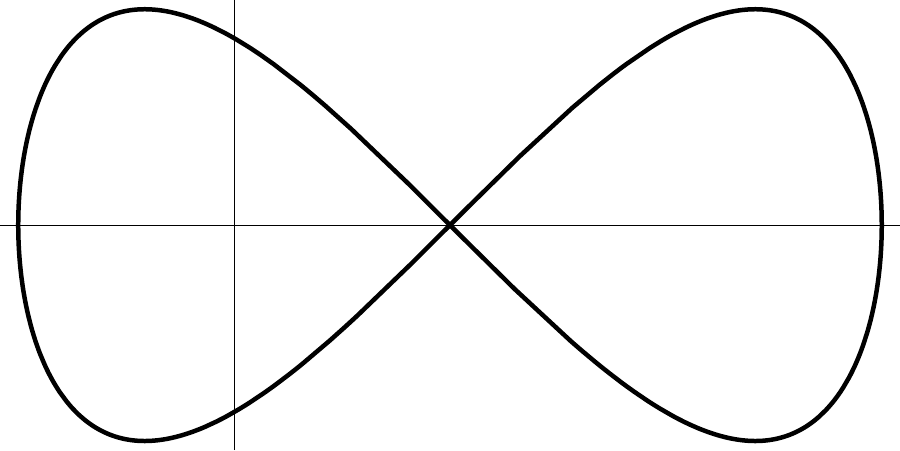}\caption{Gerono's lemniscata.}\label{gerono}
\end{center}
\end{figure}
\begin{figure}[h!]
\begin{center}
\includegraphics[height=6cm]{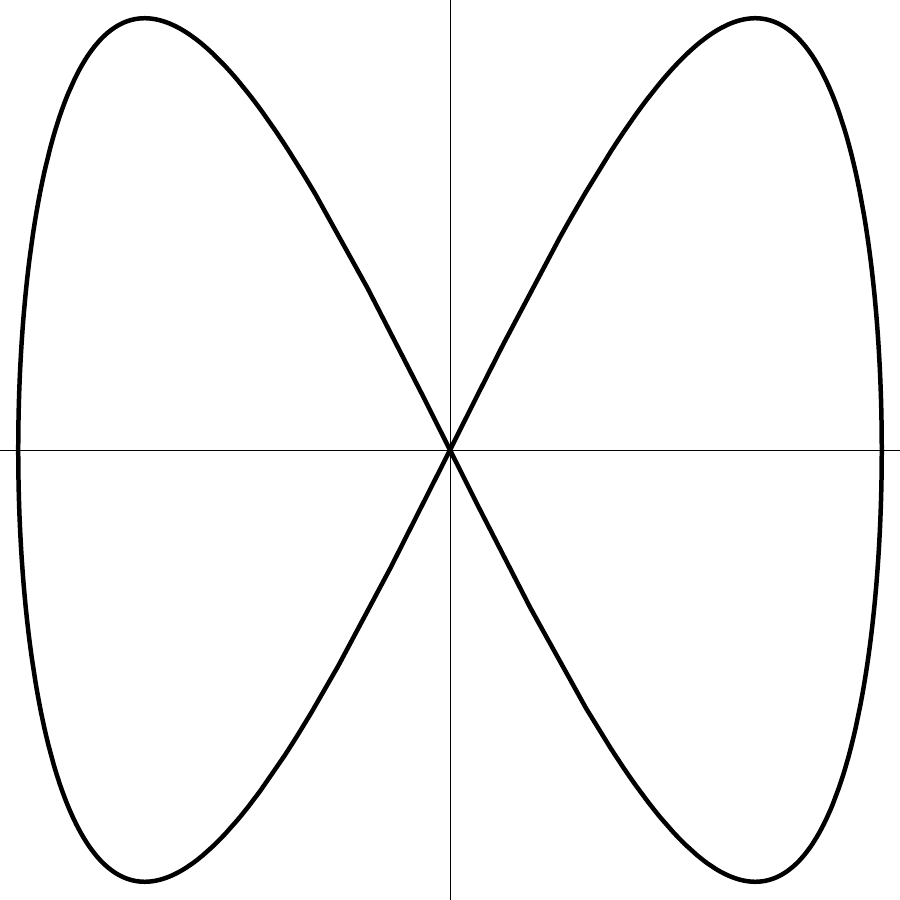}\caption{Lissajous curve.}\label{lissajous}
\end{center}
\end{figure}

\begin{center}
\begin{figure}[h!]
\begin{center}
\begin{tabular}{cc}
\includegraphics[height=5cm]{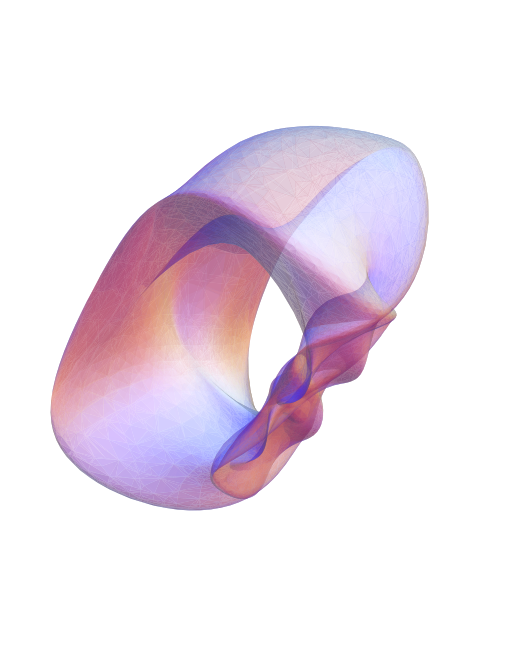} & \includegraphics[height=5cm]{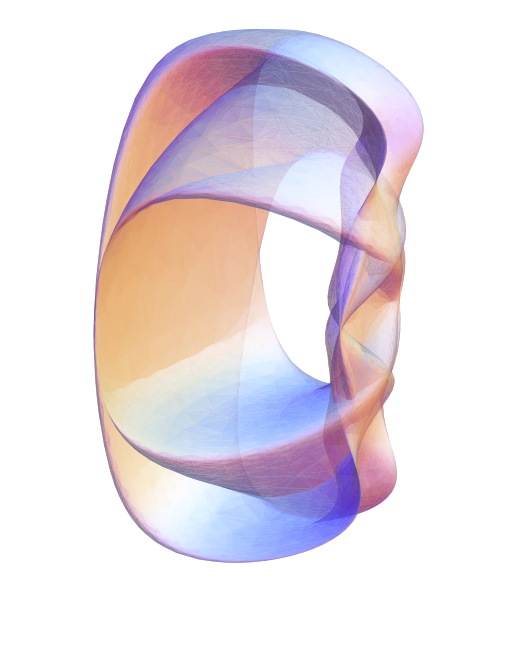}\\
\includegraphics[height=5.5cm]{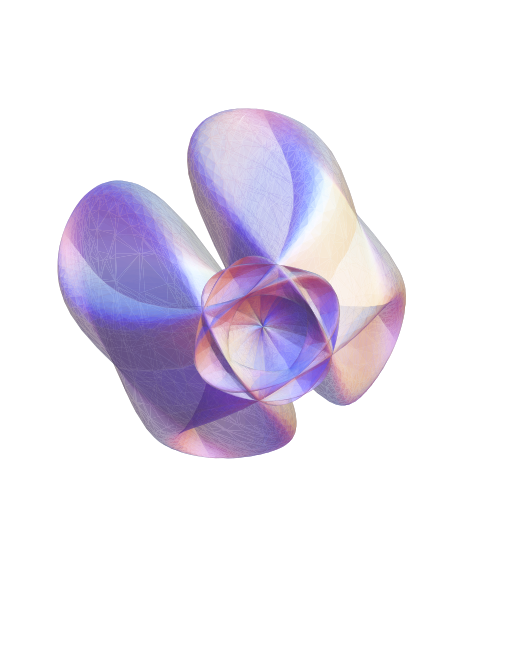} & \includegraphics[height=5cm]{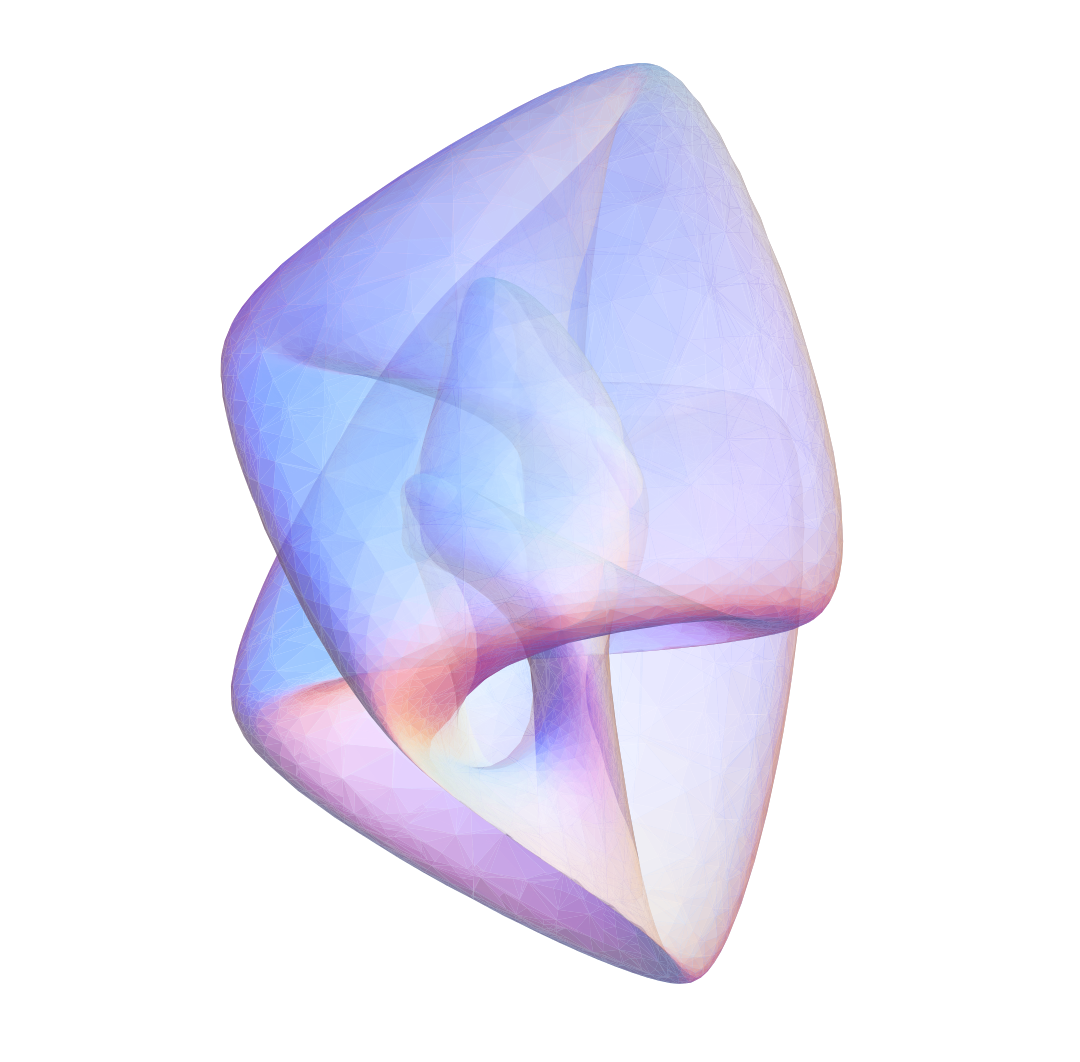}\\
\end{tabular}
\caption{Projections of the torus to the coordinate 3-spaces of $\C^2\equiv\R^4$.}\label{compact}
\end{center}
\end{figure}
\end{center}

\vspace{0.5cm}


\end{document}